\documentclass[10.5pt]{amsart}

\usepackage{verbatim,amsmath,latexsym,amssymb,amsbsy}
\usepackage[leqno]{amsmath}
\usepackage{amsthm}
\usepackage{amscd}
\usepackage{amssymb}
\usepackage{amsfonts}

\theoremstyle{plain}
\newtheorem{theorem}{Theorem}[section]
\newtheorem{corollary}[theorem]{Corollary}
\newtheorem{proposition}[theorem]{Proposition}
\newtheorem{question}[theorem]{Question}
\newtheorem{lemma}[theorem]{Lemma}
\newtheorem{remark}[theorem]{{Remark}}
\newtheorem{setup}[theorem]{Setup}
\newtheorem{definition}[theorem]{Definition}
\newtheorem{example}[theorem]{Example}

\providecommand{\lm}{\ensuremath{\lambda}}
\providecommand{\rar}{\rightarrow}
\providecommand{\lrar}{\longrightarrow}

\providecommand{\ov}{\overline}
\providecommand{\into}{\hookrightarrow}

\providecommand{\inc}{\subseteq}

\renewcommand\char{\text{\rm char}}
\providecommand\coker{\text{\rm coker}}
\providecommand\depth{\text{\rm depth}}
\renewcommand\dim{\text{\rm dim}}
\providecommand\gr{\text{\rm gr}}
\providecommand\hgt{\text{\rm ht}}
\renewcommand\Im{\text{\rm im}}
\providecommand\ker{\text{\rm ker}}
\providecommand\Min{\text{\rm Min}}
\providecommand\rank{\text{\rm rank}}

\providecommand{\Z}{{\mathbb Z}}

\providecommand\K{{{\bf K}_\bullet}}
\renewcommand\-{{\_\!\_}}
\providecommand{\sk}{{\ensuremath{\sf k }}}
\providecommand{\mq}{{\mathfrak q }}
\renewcommand{\mp}{{\mathfrak p }}
\providecommand{\m}{{\mathfrak m }}
\providecommand{\G}{{\mathcal G }}

\begin{document}

\title{\large Three-Standardness of the Maximal Ideal}

\author[H. Ananthnarayan]{H. Ananthnarayan}
\address{Department of Mathematics, University of Nebraska, Lincoln, NE 68588}
\email{ahariharan2@math.unl.edu}

\author[C. Huneke]{Craig Huneke}
\address{Department of Mathematics, University of Kansas, Lawrence, KS 66045}
\email{huneke@math.ku.edu}

\thanks{The second author was partially supported by the National Science Foundation, grant
DMS-0756853}

\date{\today}

\begin{abstract}
We study a notion called $n$-standardness (defined by M. E. Rossi in \cite{R} and extended in this 
paper) of ideals primary to the maximal ideal in a Cohen-Macaulay local ring and some of its 
consequences. We further study conditions under which the maximal ideal is three-standard, first 
proving results when the residue field has prime characteristic and then using the method of reduction to 
prime characteristic to extend the results to the equicharacteristic zero case. As an application, we 
extend a result due to T. Puthenpurakal (\cite{P}) and show that a certain length associated to a minimal 
reduction of the maximal ideal does not depend on the minimal reduction chosen.
\end{abstract}

\maketitle

\section{\large Introduction}\label{5.1}
Let $(R,\m)$ be a Cohen-Macaulay local ring of dimension $d$ with infinite residue field $\sk$. Let $I$ 
be an $\m$-primary ideal and $J = (x_1,\ldots,x_d)$ be a minimal reduction of $I$. In \cite{VV}, P. 
Valabrega and G. Valla show that the condition $I^n \cap J = J I^{n-1}$ holds for all $n$ if and only if the 
associated graded ring $\gr_R(I) = R/I \oplus I/I^2 \oplus \cdots$ is Cohen-Macaulay. 

In \cite{R}, Rossi studies the condition $J \cap I^k = J I^{k-1}$ for all $k \leq n$. We study this condition in 
a more generalized setup used by T. Marley in \cite[Chapter 3]{M} . The setup is as follows:

\begin{setup}{\label{S1}}
Let $(R,\m)$ be a Cohen-Macaulay local ring of dimension $d$ with infinite residue field $\sk$, $I$ an 
$\m$-primary ideal in $R$ and $\mathfrak F = \{I_n\}_{n \in \Z}$ be a collection of ideals of $R$ which is 
an admissible $I$-filtration, i.e., we have

\noindent 
a) $I_{n+1} \inc I_{n}$ for $n \geq 0$, $I_n = R$ for $n \leq 0$,

\noindent
b) $I_n I_m \inc I_{n+m}$ for every $n$, $m \geq 0$ and

\noindent
c) there is a $k \geq 0$ such that $I^n \inc I_n \inc I^{n-k}$ for every $n \geq 0$. (Note that this forces $I_n 
= R$ for $n \leq 0$). 

One can define the graded ring associated to the filtration $\mathfrak F$, $\gr_{\mathfrak F}(R) = 
\bigoplus_{n \geq 0} (I_n/I_{n+1})$. If $\mathfrak F = \{I^n\}_{n \in \Z}$, the standard I-adic filtration, then 
we denote its associated graded ring by $\gr_I(R)$.
\end{setup}

\noindent
We record a key observation of Marley in the following remark.

\begin{remark}{\rm
With notations as in Setup \ref{S1}, by Lemma 3.3 in (\cite{M}), if $J = (x_1,\ldots,x_d)$ is a minimal 
reduction of $I$, then $x_i \in I_1\setminus I_2$ for $i = 1, \ldots, d$ and $JI_n = I_{n+1}$ for $n >> 0$.
}\end{remark}

We use the same terminology as Rossi, in particular we define the following:

\begin{definition}\hfill{}\\
{\rm a)} Let $(R,\m)$ be a Cohen-Macaulay local ring of dimension $d$ with infinite residue field $\sk$, 
$I$ an $\m$-primary ideal and $J$ be a minimal reduction of $I$. We say that an admissible filtration $
\mathfrak F = \{I_n\}_{n \in \Z}$ is $n$-standard with respect to $J$ if $J \cap I_k = J I_{k-1}$ for all $k \leq 
n$.

\noindent
{\rm b)} We say is $\mathfrak F = \{I_n\}_{n \in \Z}$ is $n$-standard if $\mathfrak F$ is $n$-standard with 
respect to every minimal reduction $J$ of $I$.

\noindent
{\rm c)} We say that $I$ is $n$-standard with respect to $J$ (respectively $n$-standard) if the $I$-adic 
filtration is $n$-standard with respect to $J$ (respectively $n$-standard).
\end{definition}

\begin{remark}\label{R5.1}{\rm Let $(R,\m)$ be a Cohen-Macaulay local ring of dimension $d$ with 
infinite residue field $\sk$, $I$ an $\m$-primary ideal and $\mathfrak F = \{I_n\}_{n \in \Z}$ be an 
admissible $I$-filtration. Then\\
1) $\mathfrak F$ is one-standard and any $n$-standard filtration is $k$-standard for $1 \leq k \leq n$.

\smallskip

\noindent
2) If $\mathfrak F$ is $n$-standard with respect to $J$ and $I_{n+1} = I_nJ$, then $I_k \cap J = I_{k-1}J$ 
for all $k$. 
In particular, if in addition $\mathfrak F$ is the $I$-adic filtration, and is $n$-standard for all $n\geq 1$, 
then by a result of Valabrega and Valla \cite{VV}, $G = \gr_I(R)$ is Cohen-Macaulay. In particular, this 
proves that $I$ is $n$-standard for all $n \geq 1$ with respect to every minimal reduction $J$. 
In other words, for the $I$-adic filtration, the property of being $n$-standard for all $n\geq 1$ does not 
depend on the minimal reduction $J$.
\smallskip

\noindent
3) If $R$ is Cohen-Macaulay with an infinite residue field and $J$ is any minimal reduction of $\m$, it is 
well-known that $\m^2 \cap J = J\m$ (for example, see Proposition 8.3.3(1) in \cite{SH}). Thus the 
maximal ideal is two-standard.

\smallskip

\noindent 
4) If $R$ is Cohen-Macaulay with an infinite residue field and $J$ is any minimal reduction of an $\m$-
primary integrally closed ideal $I$, then C. Huneke (for rings containing a field), \cite[Theorem 4.7]{H1} 
and S. Itoh \cite[Theorem 1]{I}, independently proved that $I^2 \cap J = IJ$. Thus integrally closed $\m$-
primary ideals are two-standard. 
}\end{remark}

Let $(R,\m)$ denote a local Noetherian ring. We use $\lambda(\-)$ to denote the length of an $R$-
module and $\mu(\-)$ to denote the minimal number of generators of an $R$-module.  If $I$ is an $\m$-
primary ideal of $R$, we let $e_0(I)$ denote the multiplicity of $I$, and set  $e_0(R) = e_0(\m)$, the 
multiplicity of $R$. 

In \cite{P}, T. Puthenpurakal proved that $\lm(\m^3/J \m^2)$ is independent of the minimal reduction $J$ 
of $\m$ when $R$ is Cohen-Macaulay by proving the following theorem. 

\smallskip

\begin{theorem}[Puthenpurakal]\label{P}
Let $(R,\m,\sk)$ be a Cohen-Macaulay local ring of dimension $d \geq 1$ with infinite residue field $\sk
$. If $J$ is a minimal reduction of $\m$, then\\
\centerline{ $\lm(\m^3/J\m^2) = e_0(R) + (d-1)\mu(\m) - \mu(\m^2) - \binom{d-1}{2}.$}  
\end{theorem}

In Section \ref{6.1}, we extend this result in several ways to $n$-standard admissible $I$-filtrations. We 
go over the properties of Koszul complexes and homology needed for this purpose in Section \ref{5.2}. 
We first prove the equivalence of $n$-standardness to the vanishing of a certain Koszul homology 
module up to a certain degree in Proposition \ref{P5.2} and then use these to prove Theorem \ref{T5.1}. Combining Remark \ref{R5.1}(4) with Theorem \ref{T5.1}, we obtain the following extension of Puthenpurakal's Theorem (Theorem \ref{P}) as an immediate corollary.

\begin{theorem}\label{C5.2}
Let $(R,\m,\sk)$ be a Cohen-Macaulay local ring of dimension $d \geq 1$ with infinite residue field, $I$ an integrally closed $\m$-primary ideal and $J$ a minimal reduction of $I$. Then $\lm(I^3/JI^2) = e_0(I) - \lm(I^2/I^3) + (d-1)\lm(I/I^2) -\binom{d-1}{2}\lm(R/I)$. In particular, $\lm(I^3/JI^2)$ is independent of the minimal reduction $J$ chosen.
\end{theorem}

When $(R,\m)$ is a Cohen-Macaulay local ring with infinite residue field $\sk$ and $J$ is a minimal 
reduction of $\m$, since $\m^2 \cap J = J\m$, $\m$ is three-standard if and only if $\m^3 \cap J = J\m^2$. 
In Section \ref{6.2}, we investigate conditions under which the equality $J \cap \m^3 = J\m^2$ holds for 
every minimal reduction $J$ of $\m$.

In Theorem \ref{C6.1}, we show that $\m$ is three-standard when $\char(\sk) = p > 0$ and the graded 
ring $\G$ associated to the maximal ideal $\m$ is reduced and connected in codimension one. In Remark \ref{R6}, we give an alternate proof of three-standardness of $\m$, assuming $\sk$ is perfect 
and $\G$ is a normal domain. In order 
to prove this, we borrow some tools like tight closure (e.g., see \cite{H2}) and graded absolute integral 
closure (e.g., see \cite[Section 5]{HH1}) from the world of positive characteristic. 

We then use the method of reduction to prime characteristic (e.g., see \cite{HH3}, sections 2.1 and 2.3) to 
prove in general the following theorem:

\begin{theorem}
Let $(R,\m,\sk)$ be a Cohen-Macaulay local ring of dimension $d \geq 1$ with infinite residue field $\sk
$. Assume either that the characteristic of $\sk$ is positive and the associated graded ring $\G = \gr_\m(R)$ is reduced and connected in codimension one, or that $R$ is equicharacteristic zero and $\G$ is an 
absolute domain. Then $\m$ is three-standard.
\end{theorem}

As an immediate corollary of our work on three-standardness, we extend Puthenpurakal's result in a 
different direction in Theorem \ref{invariancep} and Theorem \ref{invariance0}, which are summarized in 
Theorem~\ref{invariance} below:

\begin{theorem}\label{invariance}
Let $(R,\m,\sk)$ be a Cohen-Macaulay local ring of dimension $d \geq 1$ with infinite residue field $\sk
$. Assume either that the characteristic of $\sk$ is positive and the associated graded ring $\G = \gr_\m(R)$ is 
reduced and connected in codimension one, or that $R$ is equicharacteristic zero and $\G$ is an 
absolute domain. If $J$ is a minimal reduction of $\m$, then
$$\lm(\m^4/J\m^3) =   e_0(R) - \mu(\m^{3}) + (d-1)\mu(\m^{2}) - \binom{d-1}{2}\mu(\m) + \binom{d-1}{3}.$$
Consequently, $\lm(\m^4/J\m^3)$ is independent of the reduction $J$.
\end{theorem}

\section{Preliminaries}\label{5.2}

\centerline{\bf Koszul Homology}
\vskip 5 pt

Let $G = \oplus_{i\geq 0} G_i$ be a graded ring with $x_1,\ldots,x_d \in G_1$. Let $\K(x_1,\ldots,x_k;G)$ 
be the Koszul complex on $x_1,\ldots,x_k$ over $G$. Then $\K(x_1,\ldots,x_k;G)$ is:
$$0 \rar G[-k] \rar G[-k+1]^{\oplus d} \rar \cdots \rar G[-2]^{\oplus \binom{d}{2}} \rar G[-1]^{\oplus d} 
\overset{(x_1,\ldots,x_k)}\lrar G\rar 0$$

\begin{remark}\label{R5.2}{\rm \hfill{}\\
1. There is a short exact sequence of complexes $$0\lrar \K(x_1,\ldots,x_{k-1};G) \lrar \K
(x_1,\ldots,x_k;G) \lrar \K(x_1,\ldots,x_{k-1};G)[-1]\lrar 0.$$

\noindent
2. Let $H_i(\-)$ be the $i$th homology in the Koszul complex. The above short exact sequence of Koszul 
complexes gives a long exact sequence on the Koszul homologies:
$$H_i(x_1,\ldots,x_{k-1})[-1] \overset{\cdot x_k}\rar H_i(x_1,\ldots,x_{k-1}) \rar H_i(x_1,\ldots,x_k) \rar H_
{i-1}(x_1,\ldots,x_{k-1})[-1] \overset{\cdot x_k}\rar $$

which breaks up into a long exact sequence of graded pieces:
$$H_i(x_1,\ldots,x_{k-1})_{j-1} \overset{\cdot x_{k}}\lrar H_i(x_1,\ldots,x_{k-1})_j \rar H_i(x_1,\ldots,x_k)_j 
\rar H_{i-1}(x_1,\ldots,x_{k-1})_{j-1}\rar...$$

\noindent
3. Notice that the $i$th Koszul homology $H_i(x_1,\ldots,x_k;G)$ is a subquotient of $G[-i]^{\oplus 
\binom{k}{i}}$. Thus if the image of $(r_1,\ldots,r_{\binom{k}{i}})$ is in $H_i(x_1,\ldots,x_k;G)_j$, then 
without loss of generality $\deg(r_l) = j-i$ as an element of $G$. \\

\noindent
4. By (3), we see that $H_i(\-;G)_j = 0$ for $j < i$. Clearly, we also have $H_i(x_1,\ldots,x_k;G) = 0$ for $i 
> k$. \\

\noindent
5. The element $\ov{(r_1,\ldots,r_k)}$ is zero in $H_1(x_1,\ldots,x_k;G)$ if it can be written as a linear 
combination of the Koszul relations, i.e.,  as elements in $G^{\oplus k}$, 

\[\left( \begin{array}{c}r_1 \\ \vdots \\ r_i \\ \vdots \\ r_j \\ \vdots \\ r_k \end{array} \right) = 
\sum_{1 \leq i < j \leq k} s_{ij} \left( \begin{array}{c}0 \\ \vdots \\ x_j \\ \vdots \\ -x_i \\ \vdots \\ 0 \end{array}
\right) \text{ where } s_{ij} \in G. \] Rearranging, we see that this happens if and only if
$(r_1,\ldots,r_k) = (x_1,\ldots,x_k) S$, where $S$ is the skew-symmetric matrix
\[\left( \begin{array}{cccc}0& - s_{12}& \cdots & -s_{1k}\\ s_{12}& 0 & \cdots & - s_{2k} \\ \vdots & & \ddots 
& \vdots \\
s_{1k} & s_{2k} & \cdots & 0 \end{array}\right).\]
}\end{remark}

\begin{lemma}\label{L5.1}
With notations as in the above remark, if $H_i(x_1,\ldots,x_k;G)_j = 0$ for all $k \leq n$, then $H_{i+1}
(x_1,\ldots,x_k;G)_{j +1} = 0$ for all $k \leq n$.
\end{lemma}

\begin{proof}
Notice that $H_i(x_1,\ldots,x_k;G)_j = 0$ by the hypothesis and $H_{i+1}(x_1,\ldots,x_{k-1};G)_{j+1} = 0$ 
by induction on $k$. We see from the long exact sequence of the
Koszul homologies that $H_i(x_1,\ldots,x_k;G)_j = 0$ and $H_{i+1}(x_1,\ldots,x_{k-1};G)_{j +1} = 0$ 
forces $H_{i+1}(x_1,\ldots,x_k;G)_{j +1} = 0$. 
\end{proof}

\begin{proposition}\label{P5.0}
Let $G_0$ be an Artinian local ring and $G = \oplus_{i\geq 0} G_i$ be a graded $G_0$-algebra with 
$x_1,\ldots,x_d \in G_1$. Let ($\ast_k$) be the complex $$0 \rar (G_0)^{\oplus \binom{d}{k}} \rar (G_1)^
{\oplus \binom{d}{k-1}} \rar \cdots \rar (G_{k-1})^{\oplus d} \rar 0 $$ obtained by truncating the $k$-th 
degree string
of the Koszul complex $\K(x_1,\ldots,x_d; G)$ for some $k \geq 2$. If $H_1(x_1,\ldots,x_m;G)_{<n} = 0$ 
for $1 \leq m \leq d$,
then $$\lm(H_0(\ast_k)) = \sum_{i=1}^k (-1)^{i-1}\binom{d}{i}\lm(G_{k-i})\quad \text{ for } 1 \leq k \leq n,$$ 
where $H_i(\ast_k)$ is the $i$-th homology of the complex ($\ast_k$). 
\end{proposition}

\begin{proof}
The proof follows immediately if we show that $H_i(\ast_k) = 0$ for $i > 0$. Note that $H_i(\ast_k) = H_{i 
+ 1}(x_1,\ldots,x_d;G)_k$ for $i > 0$. Since $H_1(x_1,\ldots,x_d; G)_j = 0$ for $j < n$, $H_{i + 1}
(x_1,\ldots,x_d; G)_{j + i} = 0$ for $j < n$ by Lemma \ref{L5.1}. In particular,   $H_{i+1}(x_1,\ldots,x_d; G)
_k = 0$ for each $i \geq 1$, proving the proposition.
\end{proof}

\begin{proposition}\label{P5.1}
Let $G = \oplus_{i\geq 0} G_i$ be a graded ring with $x_1,\ldots,x_d \in G_1$. Then with notations as 
above, $H_1(x_1,\ldots,x_k;G)_{\leq n} = 0$
for $1 \leq k \leq l$ if and only if $$(x_1,\ldots,x_{k-1}):x_k \inc (x_1,\ldots,x_{k-1}) + \oplus_{i \geq n} G_i$$ for $1 \leq k \leq l.$
\end{proposition}

\begin{proof}
First assume that $(x_1,\ldots,x_{k-1}): x_k \inc (x_1,\ldots,x_{k-1}) + \oplus_{i \geq n} G_i$ for $1 \leq k 
\leq l$. We want to prove
that $H_1(x_1,\ldots,x_k;G)_{\leq n} = 0$ by induction on $k$. 

When $k = 1$, we see that $(0:_G x_1) \inc \oplus_{i \geq n} G_i$. Note that $H_1(x_1;G) \simeq (0:_G
x_1)[-1]$. Hence $H_1(x_1;G)_{\leq n} = 0$. 

Let $\ov{(r_1,\ldots,r_k)} \in H_1(x_1,\ldots,x_k;G)_j$ for some $j \leq n$, where $k > 1$. Thus we have 
$r_i \in G_{j-1}$ and $\sum_{i=1}^k r_i x_i = 0$ in $G_j$. Thus $r_k \in (x_1,\ldots,x_{k-1}):_G x_k \inc 
(x_1,\ldots,x_{k-1}) + \oplus_{i \geq n}G_i$ by assumption. Thus $r_k = \sum_{i = 1}^{k-1} s_ix_i + s_k$ 
where $s_k \in \oplus_{i \geq n}G_i$. By degree arguments, we may assume that $s_i \in G_{j-2}$, $i = 
1, \ldots, k-1$ and $s_k = 0$ in $G$. 

Thus $0 = \sum_{i=1}^{k-1} (r_i + s_ix_k)x_i$. By induction $\ov{(r_1,\ldots,r_{k-1})} + \ov
{(x_ks_1,\ldots,x_ks_{k-1})} = 0$ in $H_1(x_1,\ldots,x_{k-1};G)$, i.e., by Remark \ref{R5.2}(6), there is a $
(k-1) \times (k-1)$ skew-symmetric matrix $S$ with entries in $G$ such that $(r_1,\ldots,r_{k-1}) + 
(x_ks_1,\ldots,x_ks_{k-1}) = (x_1,\ldots,x_{k-1}) S$. We also know that $r_k = \sum_{i = 1}^{k-1} s_ix_i$. 
Hence we have 
\[
\left( \begin{array}{c}r_1 \\ \vdots \\ \vdots \\ r_k \end{array} \right) 
= \left( \begin{array}{ccc|c}& & & -s_1\\ & S & & - s_{2} \\ & & & \\ \hline s_{1} & s_{2} & \cdots & 0 \end{array}\right)
\left( \begin{array}{c}x_1 \\ \vdots \\  \vdots \\ x_k \end{array} \right).
\]
which shows by Remark \ref{R5.2}(5) that $\ov{(r_1,\ldots,r_k)} = 0$ in $H_1(x_1,\ldots,x_k;G)$.

Conversely, let $r_k \in (x_1,\ldots,x_{k-1}): x_k$. Without loss of generality, we may assume $r_k \in G_j$ for some $j$. Write $r_kx_k = -\sum_{i = 0}^{k-1}r_ix_i$, for $r_i \in G_j$, $i = 1,\ldots,k - 1$. Thus $\overline{(r_1,\ldots,r_k)} \in H_1(x_1,\ldots,x_k; G)_{j+1}$. 

If $j \geq n$, there is nothing to prove. If $j \leq n - 1$, since $H_1(x_1,\ldots,x_k;G)_{\leq n} = 0$, we can write $(r_1,\ldots,r_k) = (x_1,\ldots,x_k) S$, where $S$ is the skew-symmetric matrix with entries in $G$. In particular, $r_k \in (x_1,\ldots,x_{k-1})$, finishing the proof.
\end{proof}

\vskip 5 pt

With  the setup as in Setup \ref{S1}, let $\G = \gr_{\mathfrak F}(R) = R/I_1 \oplus I_1/I_2 \oplus \cdots$ be 
the associated graded ring of the admissible $I$-filtration $\mathfrak F = \{I_n\}_{n\geq 0}$. If $s \in R$ is 
an element such that $s \in I_k \setminus I_{k+1}$,  we let $s'$ denote $s + I_{k+1}$, the leading form of 
$s$ in $\G$. Let $J = (x_1,\ldots,x_d)$ be a minimal reduction of $I$. 

\begin{proposition}\label{P5.2}
Let $(R,\m,\sk)$ be a Cohen-Macaulay local ring, $I$ an $\m$-primary ideal and $J = (x_1,\ldots,x_d)$ a 
minimal reduction of $I$, $\mathfrak F = \{I_n\}_{n\geq 0}$ an admissible $I$-filtration and $\G = \gr_
{\mathfrak F}(R)$ be the graded ring associated to $\mathfrak F$. With notations as in the discussion 
above, $\mathfrak F$ is $n$-standard with respect to $J$ if and only if $H_1(x'_1,\ldots,x'_k; \G)_j = 0$ for $j < n$ and $1 \leq k \leq d$.
\end{proposition}

\begin{proof}

\noindent 
Assume that $\mathfrak F$ is $n$-standard with respect to $J$. Suppose that for $j < n$ and
$0 \leq k \leq d$, $\ov{(r'_1,\ldots,r'_k)} \in H_1(x'_1,\ldots,x'_k; \G)_j$, i.e., $\sum_{i = 1}^l r'_i x'_i = 0$ in 
$\G$,
where $\deg(r'_i) = j-1$. Thus $\sum r_i x_i \in I^{j+1} \cap J = J I_j$, i.e., we can write $\sum_{i=1}^k r_i 
x_i = \sum_{i=1}^d s_i x_i$, where $s_i \in I_j$. Thus there is a skew-symmetric $k \times k$ matrix $S_k
$ such that $$(r_1,\ldots,r_k) = (x_1,\ldots,x_k) S_k + (s_1,\ldots,s_k).$$

Thus $(r'_1,\ldots,r'_k) = (x'_1,\ldots,x'_k) S'_k$ in $\G_{j-1}^{\oplus d}$, which means that $\ov{(r'_1,\ldots,r'_k)} = 0$ proving that $H_1(x'_1,\ldots,x'_k;\G)_j = 0$ for $ j < n$.

Conversely, suppose $\sum_{i=1}^k r_ix_i \in I_j$ for some $j \leq n$, with $r_i \notin I_{j-1}$. Then $
\sum_{i=1}^k r'_ix'_i = 0$ in $\G_{\leq j-1}$. Thus $\ov{(r'_1,\ldots,r'_k)} \in H_1(x'_1,\ldots,x'_k;\G)_{\leq 
j-1} = 0$, i.e., there is a skew-symmetric $k \times k$ matrix $S_k$ with entries in $R$ such that $
(r'_1,\ldots,r'_k) = (x'_1,\ldots,x'_k) S'_k$ in $\G_{\leq j-1}^{\oplus d}$, i.e., $(r_1,\ldots,r_k) = 
(x_1,\ldots,x_k) S_k+ (s_1,\ldots,s_k)$ for some $s_i \in I_{j-1}$. Thus $\sum r_ix_i = \sum s_ix_i \in JI_
{j-1}$, for each $j \leq n$, i.e., $\mathfrak F$ is $n$-standard with respect to $J$.
\end{proof}

As a consequence, we get an extension of a theorem of Valabrega and Valla \cite[Theorem 2.3]{VV}.
\begin{corollary}
Let $(R,\m,\sk)$ be a Cohen-Macaulay local ring, $I$ an $\m$-primary ideal, $J = (x_1,\ldots,x_d)$ a 
minimal reduction of $I$,  $\mathfrak F = \{I_n\}_{n\geq 0}$ an admissible $I$-filtration and $\G = \gr_
{\mathfrak F}(R)$ be the graded ring associated to $\mathfrak F$. With notations as in the discussion 
above, $x'_1,\ldots,x'_d$ is a regular sequence in $\G$ (and hence $\G$ is Cohen-Macaulay) if and only 
if $I_n \cap J = I_{n-1}J$ for all $n$.
\end{corollary}

\begin{corollary}\label{C5.1}
Let $(R,\m,\sk)$ be a Cohen-Macaulay local ring with infinite residue field, $I$ an integrally closed $\m$-
primary ideal and  $J = (x_1,\ldots,x_d)$ be a minimal reduction of $I$. With notations as above, $H_i
(x'_1,\ldots,x'_k;\G)_i = 0$ for all $k$. 
\end{corollary}

\begin{proof}
As noted before in Remark \ref{R5.1}(4), $J \cap I^2 = JI$. Hence by the Proposition \ref{P5.2}, $H_1
(x'_1,\ldots,x'_k;\G)_1 = 0$. The corollary follows by repeated application of the Lemma \ref{L5.1}. 
\end{proof}

\section{Invariance of a Length Associated to Minimal Reductions}\label{6.1}

A general question to ask is: Let $(R,\m,\sk)$ be a Cohen-Macaulay local ring with infinite residue field, 
$I$ an $\m$-primary ideal,
$J$ a minimal reduction of $I$, and $\mathfrak F = \{I_n\}_{n \in \Z}$ an admissible $I$-filtration.
Is $\lm(I_n/JI_{n-1})$ independent of the minimal reduction $J$ chosen?

In \cite[Theorem 3.6]{M}, Marley proves that the above question has a positive answer if one assumes 
that $\depth(\gr_{\mathfrak F}(R)) \geq \dim(R) - 1$. 
In fact, he proves that in this case $\lm(I_{n+1}/JI_{n}) = e_0(I) + \sum_{i=0}^n(-1)^{i+1}\binom{d-1}{i}\lm_
{\mathfrak F}(n-i)$ for all $n$,
 where $\lm_{\mathfrak F}(j)$ denotes $\lm(I_{j}/I_{j+1})$ for the rest of this section. This leads to a more 
specific question: 

\begin{question}\label{Q6.1}{\rm
Let $(R,\m,\sk)$ be a Cohen-Macaulay local ring with infinite residue field, $I$ an $\m$-primary ideal and $\mathfrak F = \{I_n\}_{n \in \Z}$ an admissible $I$-filtration. Given a minimal reduction $J$ of $I$, when is it true that
$$\lm(I_{n+1}/JI_{n}) = e_0(I) + \sum_{i=0}^n(-1)^{i+1}\binom{d-1}{i}\lm_{\mathfrak F}(n-i),$$ where $\lm_{\mathfrak F}(j)$ denotes $\lm(I_{j}/I_{j+1})$?
}\end{question}

\begin{remark}\label{R6.1}{\rm As observed before, by Theorem 3.6 in \cite{M}, Question \ref{Q6.1} has a positive answer for all n when $\depth(\gr_{\mathfrak F}(R)) \geq \dim(R) - 1$.
In particular, Question~\ref{Q6.1} is true for one-dimensional Cohen-Macaulay local rings.
}\end{remark}

In his paper \cite[Example 2]{P}, Puthenpurakal gives the following example which shows that the above formula doesn't hold in general. 

\begin{example}{\rm
Let $R = \sk[|x,y|]$, $I = (x^7,x^6y,x^2y^5,y^7)$ and $J = (x^7,y^7)$. In this case, $d = \dim(R) = 2$. One can use a computer algebra package (we use Macaulay 2) to see that $\lm(I^3/I^2J) = 3$ whereas $e_0(I) + \lm(I/I^2) - \lm(I^2/I^3) = 1$. Thus the above formula does not hold even for the $I$-adic filtration when $n = 2$ in dimension $2$.

\noindent
Note that in this case $I^2 \cap J \neq IJ$. Thus $I$ is not two-standard with respect to $J$. 
}\end{example}

In Theorem \ref{T5.1}, we show that $\lm(I_{n+1}/JI_{n})$ is independent of the minimal reduction chosen when the admissible $I$-filtration $\mathfrak F = \{I_n\}_{n \in \Z}$ is $n$-standard with respect to $J$ for every minimal reduction $J$ of $I$ by proving that Question \ref{Q6.1} has a positive answer in this case.

\begin{proposition}\label{P5.3}
Let $(R,\m,\sk)$ be a Cohen-Macaulay local ring, $I$ an $\m$-primary ideal, $\mathfrak F = \{I_n\}_{n \in \Z}$ an admissible $I$-filtration and $J$ be a minimal reduction of $I$. If $\mathfrak F$ is $n$-standard with respect to $J$, then for $1\leq k \leq n$, $\lm(JI_{k-1}/JI_k)$ $= d\lm(I_{k-1}/I_{k}) - \binom{d}{2}\lm(I_{k-2}/I_{k-1}) + \cdots + (-1)^{k-1}\binom{d}{k}\lm(R/I_1),$$$\text{i.e., }\lm(JI_{k-1}/JI_k) = \sum_{i=1}^k (-1)^{i-1}\binom{d}{i}\lm_{\mathfrak F}(k-i)\quad \text{ for } 1 \leq k \leq n.$$
\end{proposition}

\begin{proof}
Let $J = (x_1,\ldots,x_d)$ and let $\G = \gr_{\mathfrak F}(R)$. Then $x'_1,\ldots,x'_d \in \G_1$ is a system of parameters in $\G$. Now, since $R$ is Cohen-Macaulay, $x_1,\ldots,x_d$ is a regular system of parameters (for example, by Corollary 8.3.9 in \cite{SH}). We prove the proposition by induction on $k$.

For $k = 1$, consider the complex $0 \lrar (R/I_1)^{\oplus d} \overset{(x_1,\ldots,x_d)}\lrar (J/JI_1) \lrar 0$. This is clearly surjective. Injectivity 
can be seen as follows: If $\sum r_ix_i \in I_1J$, writing $\sum r_ix_i = \sum s_ix_i$ for $s_i \in I_1$, we see that there is a skew-symmetric matrix $S$ with entries in $R$ such that $(r_1,\ldots,r_d) = (x_1,\ldots,x_d)S + (s_1,\ldots,s_d)$. Since $(x_1,\ldots,x_d) \inc I_1$, we get $r_i \in I_1$, proving injectivity. 

For $k > 1$, consider the complex ($\ast_k$) as in Proposition \ref{P5.0}. In this case, since $\G_n = I_n/I_{n+1}$, we have $$0 \rar (R/I_1)^{\oplus \binom{d}{k}} \rar (I_1/I_2)^{\oplus \binom{d}{k-1}} \rar \cdots \rar (I_{k-2}/I_{k-1})^{\oplus \binom{d}{2}} \overset{\phi}\rar (I_{k-1}/I_k)^{\oplus d} \rar 0.\quad\quad(\ast_k)$$
Since $\mathfrak F$ is $n$-standard with respect to $J$, by Proposition \ref{P5.2}, $H_1(x'_1,\ldots,x'_m;\G)_{< n} = 0$ for $1 \leq m \leq d$.
Hence by Proposition \ref{P5.0}, $$d\lm(I_{k-1}/I_{k}) - \binom{d}{2}\lm(I_{k-2}/I_{k-1}) + \cdots + (-1)^{k-1}\binom{d}{k}\lm(R/I_1) = \lm(H_0(\ast_k)),$$
where $H_i(\ast_k)$ is the $i$th homology of the complex ($\ast_k$). Thus in order to prove the proposition, it is enough to prove that $H_0(\ast_k) \simeq JI_{k-1}/JI_k$, i.e., it is enough to prove that $\coker(\phi) \simeq JI_{k-1}/JI_k$. 

Consider the complex $(I_{k-2}/I_{k-1})^{\oplus \binom{d}{2}} \overset{\phi}\lrar (I_{k-1}/I_k)^{\oplus d} \overset{\psi = (x_1,\ldots,x_d)}\lrar JI^{k-1} /JI_k \lrar 0$. It is clear that $\psi$ is surjective. Therefore, to prove $\coker(\phi) \simeq JI_{k-1}/JI_k$, we need to show exactness in the middle, i.e., $\ker(\psi) = \Im(\phi)$. 

Represent the elements of $(I_{k-2}/I_{k-1})^{\oplus \binom{d}{2}} = \G_{k-2}^{\oplus \binom{d}{2}}$ as $\left( \begin{array}{cccc}0& g_{12}& \cdots & g_{1k}\\ -g_{12}& 0 & \cdots &  g_{2k} \\ \vdots & & \ddots & \vdots \\ -g_{1k} & -g_{2k} & \cdots & 0 \end{array}\right)$ where $g_{ij} \in \G_{k-2}$ for $1 \leq i < j \leq d$.
Then $$\phi\left(\left( \begin{array}{cccc}0& - g_{12}& \cdots & -g_{1k}\\ g_{12}& 0 & \cdots & - g_{2k} \\ \vdots & & \ddots & \vdots \\ g_{1k} & g_{2k} & \cdots & 0 \end{array}\right)\right) = (x'_1,\ldots,x'_d)\left( \begin{array}{cccc}0& g_{12}& \cdots & g_{1k}\\ -g_{12}& 0 & \cdots &  g_{2k} \\ \vdots & & \ddots & \vdots \\ -g_{1k} & -g_{2k} & \cdots & 0 \end{array}\right) \in \G_{k-1}^{\oplus d}.$$ 
Thus $\Im(\phi) = \{(x_1,\ldots,x_d) \cdot S\text{ where }S\text{ is a skew-symmetric matrix with entries in }\G_{k-2}\}$.

Now let $(\bar{r}_1,\ldots,\bar{r}_d) \in \ker(\psi)$, where $r_i \in I_{k-1}$ and $\bar{}$ denotes going modulo
$I_{k}$. Hence $\sum_{i = 1}^d r_ix_i = \sum_{i = 1}^d s_ix_i$ where $s_i \in I_k$. Since $x_1,\ldots,x_d$
is a regular sequence, we see that $(r_1,\ldots,r_d) = (x_1,\ldots,x_d) \cdot S + (s_1,\ldots,s_d)$ where
$S$ is a skew-symmetric matrix with entries in $R$. Thus in $\G_{k-1}$, $(\bar{r}_1,\ldots,\bar{r}_d) = (x'_1,\ldots,x'_d) \cdot S'$, where $S'$ is the skew-symmetric matrix whose entries are in $\G_{k-2}$ since $\deg(x'_i) = 1$ and $\bar {s_i} = 0$ in $\G$. Thus $\ker(\psi) = \Im(\phi)$ proving that $\coker(\phi) \simeq JI_{k-1}/JI_k$, which finishes the proof.
\end{proof}

The following theorem shows that Question \ref{Q6.1} has a positive answer when $\mathfrak F$ is $n$-standard with respect to $J$. 

\begin{theorem}\label{T5.1}
Let $(R,\m,\sk)$ be a Cohen-Macaulay local ring with an infinite residue field $\sk$. If $I$ is an $\m$-primary ideal, $\mathfrak F = \{I_n\}_{n \in \Z}$ an admissible $I$-filtration and $J$ is a minimal reduction of $I$ such that $J \cap I_k = JI_{k-1}$ for $1 \leq k \leq n$, i.e., $\mathfrak F$ is $n$-standard with respect to $J$, then for $0 \leq k \leq n$, $\lm(I_{k+1}/JI_{k}) = e_0(I) - \lm_{\mathfrak F}(k) + (d-1)\lm_{\mathfrak F}(k-1) +\cdots + (-1)^{k}\binom{d-1}{k-1}\lm_{\mathfrak F}(1) +(-1)^{k+1} \binom{d-1}{k}\lm_{\mathfrak F}(0)$, $$\text{i.e., } \lm(I_{k+1}/JI_k) = e_0(I) + \sum_{i=0}^k(-1)^{i+1}\binom{d-1}{i}\lm_{\mathfrak F}(k-i).\quad\text{ for }0 \leq k \leq n\quad(\sharp)$$
\end{theorem}

\begin{proof}
We prove this by induction on $n$. For $n = 0$, since $e_0(I) = \lm(R/J)$, $(\sharp)$ holds. 

\noindent
Assume $n \geq 1$. We have $\lm(I_{k+1}/JI_k) = e_0(I) + \sum_{i=0}^k(-1)^{i+1}\binom{d-1}{i}\lm_{\mathfrak F}(k-i)$ for $0 \leq k \leq n - 1$ by induction. Hence we need to prove ($\sharp$) only for $k = n$.

We have $\lm(I_{n+1}/JI_{n}) = $ \\ $\lm(I_{n}/JI_{n-1}) + \lm(JI_{n-1}/JI_{n}) - \lm(I_{n}/I_{n+1}) = \lm(I_{n}/JI_{n-1}) + \lm(JI_{n-1}/JI_{n}) - \lm_{\mathfrak F}(n)$. 

By induction, $\lm(I_{n}/JI_{n-1}) = e_0(I) + \sum_{i=0}^{n-1}(-1)^{i+1}\binom{d-1}{i}\lm_{\mathfrak F}(n-i-1)$ and by Proposition \ref{P5.3}, $\lm(JI_{n-1}/JI_{n}) = \sum_{i=1}^n (-1)^{i-1}\binom{d}{i}\lm_{\mathfrak F}(n-i)$ since $\mathfrak F$ is $n$-standard with respect to $J$. Combining these by using Pascal's identity $\binom{d}{i} - \binom{d-1}{i-1} = \binom{d-1}{i}$, we get ($\sharp$) for $k = n$.
\end{proof}

\begin{remark}\label{invremark} 
{\rm Let $(R,\m,\sk)$ be a Cohen-Macaulay local ring of dimension $d \geq 1$ with infinite residue field $\sk$. Assume that $\m$ is three-standard.
Then for every minimal reduction $J$ of $\m$,
$$\lm(\m^4/J\m^3) =   e_0(R) + \sum_{i=0}^3(-1)^{i+1}\binom{d-1}{i}\lm(\m^{3-i}/\m^{4-i}).$$
Consequently, $\lm(\m^4/J\m^3)$ is independent of the reduction $J$.
}\end{remark}

One can see from the following example that $n$-standardness is not necessary for a positive answer to Question \ref{Q6.1}.

\begin{example}\label{E6.1}{\rm
Let $R = \sk[[t^4,t^5,t^{11}]]$. Then $R$ is a $1$-dimensional Cohen Macaulay local ring. Consider the minimal reduction $J = (t^4)$ of $\m$,
the maximal ideal. We see that $t^{15} \in J \cap \m^3 \setminus J\m^2$, showing that $\m$ is not three-standard with respect to $J$. However, by Remark \ref{R6.1}, Question \ref{Q6.1} has a positive answer for the $\m$-adic filtration and any minimal reduction $J$ of $\m$.  
}\end{example}

\begin{remark}{\rm
Let $(R,\m,\sk)$ be a $d$-dimensional Cohen-Macaulay local ring with infinite residue field, $I$ an $\m$-primary ideal and $J$ be a minimal reduction of $I$ such that $I$ is $n$-standard with respect to $J$ for some $n \geq 1$. Further assume that $\lm(I^{n+1}/I^nJ) \leq 1$. Then by Rossi's Theorem, \cite[Theorem 3.2]{R},  we get $\depth(\gr_I(R)) \geq d - 1$. Combining this with Marley's Theorem \cite[Theorem 3.6]{M},  we can conclude that $\lm(I^{n+1}/JI^n)$ is independent of $J$ for all $n$. 
}\end{remark}

Thus $n$-standardness of $I$ with respect to $J$ together with the condition $\lm(I^{n+1}/I^nJ) \leq 1$ forces $\depth(\gr_I(R)) \geq d - 1$, which leaves us with a situation where Marley's result is applicable. We conclude this section with the following question:

\begin{question}
Let $(R,\m,\sk)$ be a Cohen-Macaulay local ring with infinite residue field, $I$ an $\m$-primary ideal and $J$ a minimal reduction of $I$. Under what other conditions does $n$-standardness force $\depth(\gr_I(R)) \geq \dim(R) - 1$?
\end{question}

\section{Three-Standardness of the Maximal Ideal: the Prime Characteristic case}\label{6.2}

If $(R,\m,\sk)$ is a Cohen-Macaulay local ring, we know that $\m$ is two-standard by Remark \ref{R5.1}(2). Example \ref{E6.1} shows that the maximal ideal is not three-standard in general. In this section, we study conditions under which the maximal is three-standard when $\char(\sk) = p > 0$. Just as in the Valabrega-Valla Theorem, the graded ring associated to the maximal ideal plays a role in the theorem.
In particular, we prove in Theorem \ref{C6.1} that if the associated graded ring is reduced and connected in codimension one, then the maximal ideal is three-standard. We use the following theorem(\cite[Theorem 5.15]{HH1}) in its proof.

\begin{theorem}[Hochster-Huneke]\label{HH}
Let $G$ be a standard graded domain over a field $\sk$ of characteristic $p > 0$. Let $\mathfrak{G}$ be 
the graded absolute integral closure of $G$. Then every sequence that is a part of a homogeneous 
system of parameters in $G$ forms a regular sequence in $\mathfrak{G}$.
\end{theorem}

In Remark \ref{R6}, we give an alternate proof of three-standardness of $\m$, assuming $\sk$ is perfect 
and $\G$ is a normal domain using the following theorem \cite[Theorem 3.1]{HV}.

\begin{theorem}[Huneke-Vraciu]\label{HV}
Let $(R,\m,\sk)$ be an excellent normal local domain where $\sk$ a perfect field of $\char(\sk) = p > 0$, 
such that $R$ has reduced associated graded ring $\gr_{\m}(R)$. If $I$ is an ideal such that $I \in \m^k$, 
then $I^* \inc I + \m^{k+1}$, where $I^*$ is the tight closure of $I$.
\end{theorem}

We first prove a general lemma which is independent of the characteristic of the residue field. We use the following definition and Remark \ref{R5.4} in Lemma \ref{L5.0}.

\begin{definition}
We say that a ring $G$ is connected in codimension one if given any two minimal primes $\mp$ and $\mq$ in $G$, there is a sequence of minimal primes $\mp = {\mp}_1, \ldots, {\mp}_k = \mq$ such that $\hgt({\mp}_i + {\mp}_{i+1}) \leq 1$.
\end{definition}

\begin{remark}{\rm
Some of our results use the hypothesis that the graded ring under consideration, say $G$,
 is connected in codimension one. It is not completely clear when this
occurs. Obviously if $G$ is a domain,
or if $G$ satisfies Serre's condition $S_2$ (see \cite{HH2}
for this statement and for other information concerning this property), it is connected in codimension one.
On the other hand, a graded ring can be reduced and not
be connected in codimension one. For example, $G = \sk[x,y,u,v]/(xu,xv,yu,yv)$
is reduced but not connected in codimension one.
}\end{remark}

\begin{remark}\label{R5.4}{\rm Let $(x_1,\ldots,x_d)$ be a permutable system of parameters in a Noetherian ring $G$. Let $S$ be a matrix with coefficients in $G$ such that $\rank(S) = r$,
and such that some $r$ by $r$ minor of $S$ is invertible. Let $I$ be the ideal generated by the $d$ components of the vector
$(x_1,\ldots,x_d)S$. Without loss of generality, we may assume that $S$ is in its reduced row echelon form. Observe that by reordering the $x$'s if necessary, $I = (y_1,\ldots,y_r)$, where $y_i$ is of the form $x_i + \sum_{j=r+1}^d a_{ij}x_j$ for $i = 1,\ldots,r$. Thus we see that $\hgt(I) = r = \rank(S)$.
}\end{remark}

\begin{lemma}\label{L5.0}
Let $G = \oplus_{i\geq 0} G_i$ be a standard graded algebra over a field $\sk$ with $x_1,\ldots,x_d \in G_1$. Suppose that $G$ is reduced and connected
in codimension one. Let $\Min(G) = \{ \mp_1,\ldots,\mp_l \}$ be the set of minimal primes of $G$. If $H_1(x_1,\ldots,x_k; G/\mp_i)_{\leq 2} = 0$ for $i = 1, \ldots, l$, then $H_1(x_1,\ldots,x_k;G)_{\leq 2} = 0$.  
\end{lemma}

\begin{proof}
Let $\ov{(r_1,\ldots,r_k)} \in H_1(x_1,\ldots,x_k;G)_{\leq 2}$, i.e., $\sum_{i=1}^k r_ix_i = 0$ in $G$ with $\deg(r_i) < 2$. Let $\ \bar{}\ $ denote going modulo $\mp_i$. Since $\ov{(\bar{r}_1,\ldots,\bar{r}_k)} \in H_1(x_1,\ldots,x_k; G/\mp_i)_{\leq 2} = 0$ for each $i$, we can write $(r_1,\ldots,r_k) = (x_1,\ldots,x_k) S_i + ({p_i}_1,\ldots,{p_i}_k)$, where $S_i$ (or we write $S_{\mp_i}$ if we need to label the prime) is a skew-symmetric matrix with entries in $G_0 = \sk$ and ${p_i}_n \in \mp_i$, $n = 1, \ldots, l$.

Now since $G$ is connected in codimension one, without loss of generality we may assume the minimal primes $\{ \mp_1,\ldots,\mp_l \}$ are ordered
so that $\hgt(\mp_i + \mp_{i+1})\leq 1$ for $1\leq i\leq l-1$.  

\noindent
\underline{Claim:} ${S_i} = {S_{i+1}}$ for all $i$.

\noindent
 
If ${S_i} \neq {S_{i+1}}$, then ${S_i}- {S_{i+1}}$ is a non-zero skew-symmetric matrix, i.e., it has a $2 \times 2$ minor of the form
$\left(\begin{array}{cc}0&r\\r&0\end{array}\right)$ where $r \neq 0$. Thus $\rank({S_i} - {S_{i+1}}) \geq 2$. 
Since $(x_1,\ldots,x_k)({S_i} - {S_{i+1}}) \in {\mp_i} + {\mp_{i+1}}$, we see by Remark \ref{R5.4} that $\rank({S_i} - {S_{i+1}}) \geq 2$ forces
$\hgt({\mp_i} + {\mp_{i+1}}) \geq 2$. This contradiction proves the claim.

Hence for any two minimal primes $\mp$ and $\mq$ we have $S_\mp = S_\mq$. Since $(p_1,\ldots,p_n) - (q_1,\ldots,q_k) = (x_1,\ldots,x_k) (S_\mp - S_\mq)$, this forces $({p_m}_1,\ldots,{p_m}_k)$ $= ({p_n}_1,\ldots,{p_n}_k)$ for $1\leq m,n \leq l$. 

Let $S_i = S$ and $({p_i}_1,\ldots,{p_i}_k) = (p_1,\ldots,p_n)$, $i = 1,\ldots, l$. Thus $p_n \in \cap_{i=1}^l \mp_i = 0$ for $n = 1, \ldots, k,$ since $G$ is reduced. Therefore we have $(r_1,\ldots,r_k) = (x_1,\ldots,x_k) S$. This proves that $\ov{(r_1,\ldots,r_k)} =0$ in $H_1(x_1,\ldots,x_k; G)$ proving the lemma.
\end{proof}

\begin{theorem}\label{C6.1}
Let $(R,\m,\sk)$ be a $d$-dimensional Cohen-Macaulay local ring of $\char(\sk) = p > 0$. Let $\G = \gr_R(\m) = R/\m \oplus \m/\m^2 \oplus \cdots$ 
be the graded ring associated to the maximal ideal. If $\G$ is reduced and connected in codimension one, then $\m$ is three-standard.
\end{theorem}

\begin{proof}
 Let $J = (x_1,\ldots,x_d)$ be a minimal reduction of $\m$. By Proposition \ref{P5.2} and Corollary \ref{C5.1}, it is enough to show that $H_1(x'_1,\ldots,x'_k; \G)_2 = 0$ for $1 \leq k \leq d$. 
This follows from the following lemma. \end{proof}

\begin{lemma}\label{T6.1}
Let $G$ be a standard graded $\sk$-algebra where $\sk$ is a field of prime characteristic $p > 0$. Let $x_1,\ldots,x_d$ be a linear system of parameters in $G$. If $G$ is reduced and connected in codimension one, then $H_1(x_1,\ldots,x_k;G)_2 = 0$ for $1 \leq k \leq d$. 
\end{lemma}
\begin{proof}
Let $\Min(G) = \{ \mp_1,\ldots,\mp_l \}$ be the set of minimal primes of $G$. By Lemma \ref{L5.0}, it is enough to show that $H_1(x_1,\ldots,x_k; G/\mp_i)_2 = 0$ for $i=1,\ldots,l$. 

Let $\mp$ be a minimal prime of $G$. Let $\overline{(r_1,\ldots,r_k)} \in  H_1(x_1,\ldots,x_k; G/\mp)_2$. Then we have $\sum_{i=1}^j r_i x_i = 0$ in $G/\mp$, where $\deg(r_i) = 1$.

Let $\mathfrak{G} = (G/\mp)^{\gr +}$ be the graded absolute integral closure of $G/\mp$. Then $\mathfrak{G}$ is a big Cohen-Macaulay $G/\mp$-algebra by Theorem \ref{HH}.
Hence $x_1,\ldots,x_d$ form a regular sequence in $\mathfrak{G}$. Therefore the only relations on $x_1,\ldots,x_k$ are the Koszul relations, i.e., we can write $(r_1,\ldots,r_k) = (x_1,\ldots,x_k) S_{k \times k}$, 
where $S$ is a $k \times k$ skew-symmetric matrix with entries in $\mathfrak{G}$. By degree arguments, we can assume that
the entries of $S$ are units in $\mathfrak{G}$, i.e., the entries of $S$ are in $\bar \sk$, an algebraic closure of $\sk$.
Thus $H_1(x_1,\ldots,x_k; G/\mp)_2 \otimes_\sk \bar \sk = 0$, which shows that $H_1(x_1,\ldots,x_k; G/\mp)_2 = 0$ finishing the proof.\footnote{We
thank Mark Walker for helping us to remove our original assumption that $\sk$ is algebraically closed.}
\end{proof}

\begin{remark}\label{R6}{\rm
When $\sk$ is perfect, $\char(\sk) = p > 0$ and $G$ is a normal domain, one can show that $H_1(x_1,\ldots,x_k;G)_{\leq 2} = 0$ for $1 \leq k \leq d$ without appealing to the absolute integral closure. This can be done as follows:

By the colon-capturing property of tight closure, (cf. \cite[Theorem 3.1]{H2}), we see that for $1 \leq k \leq d$, $(x_1,\ldots,x_{k-1}):_G x_k \inc (x_1,\ldots,x_{k-1})^*$. 

Let $\m = G_{>0}$. Now by Theorem \ref{HV}, since $(x_1,\ldots,x_{k-1}) \inc \m$ but not in  $\m^2$, we get $(x_1,\ldots,x_{k-1})^* \inc (x_1,\ldots,x_{k-1}) + \m^2$ for $1 \leq k \leq d$. 

Thus $(x_1,\ldots,x_{k-1}):_G x_k \inc(x_1,\ldots,x_{k-1}) + \m^2$ for $1 \leq k \leq d$. Hence by Proposition \ref{P5.1}, $H_1(x_1,\ldots,x_k;G)_{\leq 2} = 0$ for $1 \leq k \leq d$. 
}\end{remark}

We see in Lemma \ref{T6.1} that if $G$ is reduced and connected in codimension one, then for $1 \leq k \leq d$, $H_1(x_1,\ldots,x_k;G)_{\leq 2} = 0$.  However, it is possible that $H_1(x_1,\ldots,x_k;G)_{3} \neq 0$ even when $G$ is a normal domain as can be seen from the following example.

\begin{example}\label{E6.0}{\rm
Let $R = \sk[X,Y,Z]/(X^3+Y^3+Z^3)$, where $\sk$ is a perfect field such that $\char(\sk) \neq 3$ and let $G = R[(x,y,z)t]$ be the Rees ring associated to the homogeneous maximal ideal $\m = (x,y,z)$. Then $G$ is a normal domain. 

One can see using a computer algebra package (we use Macaulay2) that a presentation for $G$ is the following: $G \simeq \sk[X,Y,Z,U,V,W]/I$ where $I = (X^3+Y^3+Z^3,X^2U+Y^2V+Z^2W,XU^2+YV^2+ZW^2,U^3+V^3+W^3,YW-ZV,XW-ZU,XV-YU)$. We use lower case letters to denote elements of $G$. 

Consider the linear system of parameters $f_1 = x$,$f_2 = y+u$ and $f_3 = z+v$. Then $\ov{(x^2-yv+w^2,y^2-zw,z^2)} \in H_1(\underline{f})_3$ is a non-zero element, showing that $H_1(\underline{f})_3 \neq 0$.

As a consequence of Proposition \ref{P5.2}, one observes that if $G$ is the associated graded ring of a Cohen-Macaulay local ring $(R,\m)$, then $J \cap \m^4 \neq J\m^3$, where $J = (f_1,f_2,f_3)$ is a minimal reduction of $\m$ such that the leading forms of $f_1$, $f_2$ and $f_3$ in $G$ are $x$, $y+u$ and $z+v$ respectively.
}\end{example}

We can now prove the promised generalization of Puthenpurakal's theorem in positive characteristic.

\begin{theorem}\label{invariancep}
Let $(R,\m,\sk)$ be a Cohen-Macaulay local ring of dimension $d \geq 1$ with infinite residue field $\sk$. Assume that the characteristic of
$\sk$ is positive and the associated graded ring $\G = \gr_\m(R)$ is reduced and connected in codimension one.
If $J$ is a minimal reduction of $\m$, then
$$\lm(\m^4/J\m^3) =   e_0(R) + \sum_{i=0}^3(-1)^{i+1}\binom{d-1}{i}\lm(\m^{3-i}/\m^{4-i}).$$
Consequently, $\lm(\m^4/J\m^3)$ is independent of the reduction $J$.
\end{theorem}

\begin{proof} This is immediate from Theorem~\ref{C6.1} together with Remark~\ref{invremark}. \end{proof}

\section{The Characteristic Zero Case: Reduction to Prime Characteristic}\label{6.4}

In this section, we prove an analogue of Lemma \ref{T6.1} in the case when the residue field has characteristic zero. We use the method 
 of reduction to characteristic $p$. Our main source for this technique are sections 2.1 and 2.3 of \cite{HH3}.

\noindent
We begin by recalling the following definition (Definition 2.3.1) from \cite{HH3}.
\begin{definition}
We say that a \sk-algebra $R$ is an absolute domain if $R \otimes_{\sk} \bar \sk$ is a domain, where $\bar \sk$ is the algebraic closure of $\sk$. We say that a prime ideal $\mp \inc R$ is an absolute prime if $R/\mp$ is an absolute domain.
\end{definition}

\begin{remark}{\rm
By definition, any graded domain over an algebraically closed field is an absolute domain. Furthermore, if $G$ is an absolute domain of equicharacteristic zero, then by Theorem 2.3.6(c) in \cite{HH3}, almost all the graded rings obtained from $G$ by the process of reduction to prime characteristic are also absolute domains.
}\end{remark}

\begin{setup}\label{S6.1}
Let $\sk$ be a field of characteristic 0. Let $G$ be a standard graded $\sk$-algebra and $x_1,\ldots,x_d \in G$ be a linear system of parameters such that for each $k = 1,\ldots,d$, $$(x_1,\ldots,x_{k-1}):_G x_k \inc (x_1,\ldots,x_{k-1}) + G_{\geq n}.$$
\end{setup}

We will now apply the method of reduction to prime characteristic to this setup. The following lemma \cite[2.1.4]{HH3} plays a key role in this process.

\begin{lemma}[Generic Freeness]
Let $A$ be a Noetherian domain, $R$ a finitely generated $A$-algebra, $S$ a finitely generated $R$-algebra, $W$ a finitely generated $S$-module, $M$ a finitely generated $R$-submodule of $W$ and $N$ a finitely generated $A$-submodule of $W$. Let $V = W/(M+N)$. Then there exists an element $a \in A \setminus \{0\}$ such that $V_a$ is free over $A_a$.
\end{lemma}

\noindent
Write $G \simeq \sk[X_1,\ldots,X_m]/(F_1,\ldots,F_n)$ where $X_i \mapsto x_i$ for $i = 1,\ldots,d \leq m$. Write $A = \Z[\text{coefficients of the }F_j\text{'s}]$.

Let $G_A = A[X_1,\ldots,X_m]/(F_1,\ldots,F_n)$. By the lemma of Generic Freeness, after inverting an element $a$ of $A$, and replacing $A_a$ by $A$, we may assume that $G_A$ is a free $A$-module. 

Since $G_A$ is a free $A$-module, the inclusion $A \into \sk$ induces the injective map $G_A \into G_{\sk} := G_A \otimes_A \sk$.\footnote{We only need that $G_A$ is $A$-flat.} Further, we see that $G \simeq \sk[\underline{X}] \otimes_{A[\underline{X}]}A[\underline{X}]/(\underline{F}) \simeq (\sk \otimes_A A[\underline{X}]) \otimes_{A[\underline{X}]}A[\underline{X}]/(\underline{F}) \simeq \sk \otimes_{A}A[\underline{X}]/(\underline{F})$.

By further inverting another element of $A$ if necessary (and calling the localization $A$ again), we see by \cite{HH3}, 2.1.14(a)-(c),(g) that for each $k = 1,\ldots,d$, $(x_1,\ldots,x_{k-1}):_{G_A} x_k \inc (x_1,\ldots,x_{k-1}) + (G_A)_{\geq n}.$

Let $\m_A$ be any maximal ideal in $A$. Then there is some prime $p  \in \m_A$. Thus if $G' = G_A/\m_AG_A$, we see that $G'$ is a standard graded $\sk'$-algebra, where $\sk'$ is a field of characteristic $p > 0$. We say that {\it $G$ descends to $G'$} or that {\it $G'$ descends from $G$}.

Let $x_i'$ denote the image of $x_i$ in $G'$. Notice that each $x_i'$ is a linear form in $G'$. Now, by Theorem 2.3.5(c) in \cite{HH3}, we see that $\dim(G) = \dim(G')$, hence $x_1',\ldots, x_d'$ form a linear system of parameters in $G'$. The condition that $(x'_1,\ldots,x'_{k-1}):_{G'} x'_k \inc (x'_1,\ldots,x'_{k-1}) + (G')_{\geq n}$ holds for each $k = 1,\ldots,d$ for all but finitely many maximal ideals $\m_A \in A$ by Theorem 2.3.5(g) in \cite{HH3}. Choose an $\m_A$ such that $(x'_1,\ldots,x'_{k-1}):_{G'} x'_k \inc (x'_1,\ldots,x'_{k-1}) + (G')_{\geq n}$ holds for each $k = 1,\ldots,d$. 

Suppose further that $G$ is an absolute domain. By Theorem 2.3.6(c) in \cite{HH3}, we see that for all but finitely many maximal ideals $\m_A$ in $A$, $G' = G_A/\m_A$ is an absolute domain. Choosing one such $\m_A$ for which the condition $(x'_1,\ldots,x'_{k-1}):_{G'} x'_k \inc (x'_1,\ldots,x'_{k-1}) + (G')_{\geq n}$ also holds for each $k = 1,\ldots,d$ we see that:

\begin{theorem}\label{preduction}
Let the notation be as in Setup \ref{S6.1}. Suppose $G$ is an absolute domain. Then there is a field $\sk'$ of prime characteristic, an absolute domain $G'$ which is a standard graded $\sk'$-algebra, $x'_1,\ldots,x'_d$, a linear system of parameters in $G'$ satisfying $(x'_1,\ldots,x'_{k-1}):_{G'} x'_k \inc (x'_1,\ldots,x'_{k-1}) + (G')_{\geq n}$ for each $k = 1,\ldots,d$ such that $G$ descends to $G'$.
\end{theorem}

\begin{proposition}\label{T6.3}
Let $G$ be a standard graded algebra over a field $\sk$ with $\char(\sk) = 0$. Let $x_1,\ldots,x_d$ be a linear system of parameters in $G$. If $G$ is an absolute domain, then for $1 \leq k \leq d$, $H_1(x_1,\ldots,x_k;G)_2 = 0$.
\end{proposition}
\begin{proof} Suppose $H_1(x_1,\ldots,x_k; G)_2 \neq 0$ for some $k$, $1 \leq k \leq d$. By Proposition \ref{P5.1} and Theorem \ref{preduction}, there is a field $\sk'$ of some prime characteristic $p > 0$, a standard graded $\sk'$-algebra $G'$ which is an absolute domain with a system of parameters $x_1,\ldots,x_d$ such that $H_1(x_1,\ldots,x_k; G')_2 \neq 0$. This contradicts Lemma \ref{T6.1}.
\end{proof}

As a consequence of Proposition \ref{T6.3}, Proposition \ref{P5.2} and Corollary \ref{C5.1}, we conclude that:

\begin{theorem}\label{C6.3}
If $(R,\m,\sk)$ is a $d$-dimensional Cohen-Macaulay local ring with associated graded ring $\G = \gr_R(\m) = R/\m \oplus \m/\m^2 \oplus \cdots$, then $\m$ is three-standard when $R$ is equicharacteristic zero and $\G$ is an absolute domain.
\end{theorem}

Our last theorem gives the generalization of Puthenpurakal's theorem in equicharacteristic $0$.

\begin{theorem}\label{invariance0}
Let $(R,\m,\sk)$ be a Cohen-Macaulay local ring of dimension $d \geq 1$ with infinite residue field $\sk$. Assume that
$R$ is equicharacteristic zero and
$\G$ is an absolute domain. If $J$ is a minimal reduction of $\m$, then
$$\lm(\m^4/J\m^3) =   e_0(R) + \sum_{i=0}^3(-1)^{i+1}\binom{d-1}{i}\lm(\m^{3-i}/\m^{4-i}).$$
Consequently, $\lm(\m^4/J\m^3)$ is independent of the reduction $J$.
\end{theorem}

\begin{proof} This is immediate from Theorem~\ref{C6.3} together with Remark~\ref{invremark}. \end{proof}

We conclude with the following question:

\begin{question}
Let $(R,\m,\sk)$ be a Cohen-Macaulay local ring of dimension $d \geq 1$. When is an $\m$-primary integrally closed ideal three-standard?
\end{question}

\bibliographystyle{amsplain}

\end{document}